\documentclass{birkjour}
 \newtheorem{thm}{Theorem}[section]
 
 \newtheorem{lem}[thm]{Lemma}
 \newtheorem{prop}[thm]{Proposition}
 \theoremstyle{definition}
 
 \theoremstyle{remark}

 \numberwithin{equation}{section}

\newcommand{\dv}{\mathrm{div}}
\newcommand{\tr}{\mathrm{tr}}
\newcommand{\ie}{i.e. }
\newcommand{\X}{\mathfrak{X}}
\newcommand{\W}{\mathcal{W}}

\newcommand{\R}{\mathbb{R}}

\newcommand{\n}{\nabla}

\newcommand{\D}{{\rm d}}

\newcommand{\al}{\alpha}

\newcommand{\lm}{\lambda}
\newcommand{\ta}{\theta}
\newcommand{\ep}{\varepsilon}

\newcommand{\sx}{\mathop{\mathfrak{S}}\limits_{x,y,z}}
\newcommand{\thmref}[1]{The\-o\-rem~\ref{#1}}

\newcommand{\lemref}[1]{Lem\-ma~\ref{#1}}

\begin{document}

%
%
%
%
%
%
%
%
%

\title[Conformal Riemannian $P$-Manifolds]
 {Conformal Riemannian $P$-Manifolds \\
 with Connections
whose Curvature Tensors are Riemannian $P$-Tensors }

\author[D. Gribacheva]{Dobrinka Gribacheva}

\address{%
Bulgaria Blvd 236\br Department of Algebra and Geometry\br Faculty
of Mathematics and Informatics\br University of Plovdiv\br 4003
Plovdiv\br Bulgaria}

\email{dobrinka@uni-plovdiv.bg}

\thanks{This work was partially supported by project
NI11-FMI-004 of the Scientific Research Fund, Paisii Hilendarski
University of Plovdiv, Bulgaria.}
\author[D. Mekerov]{Dimitar Mekerov}
\address{Bulgaria Blvd 236\br
Department of Algebra and Geometry\br Faculty of Mathematics and
Informatics\br University of Plovdiv\br 4003 Plovdiv\br Bulgaria}
\email{mircho@uni-plovdiv.bg}



\begin{abstract}
The largest class of Riemannian almost product manifolds, which is
closed with respect to the group of the conformal transformations
of the Riemannian metric, is the class of the conformal Riemannian
$P$-manifolds. This class is an analogue of the class of the
conformal K\"ahler manifolds in almost Hermitian geometry. The
main aim of this work is to obtain properties of manifolds of this
class with connections, whose curvature tensors have similar
properties as the K\"ahler tensors in Hermitian geometry.
\end{abstract}

\maketitle
\section{Introduction}

A Riemannian almost product manifold $(M, P, g)$ is a
differentiable manifold $M$ for which almost product structure $P$
is compatible with the Riemannian metric $g$ such that an isometry
is induced in any tangent space of $M$. The systematic development
of the theory of Riemannian almost product manifolds was started
by K. Yano in \cite{1}, where basic facts of the differential
geometry of these manifolds are given. In \cite{2} A. M. Naveira
gave a classification of Riemannian almost product manifolds with
respect to the covariant differentiation $\n P$, where $\n$  is
the Levi-Civita connection of $g$. This classification is very
similar to the Gray-Hervella classification in \cite{3} of almost
Hermitian manifolds. Having in mind the results in \cite{2}, M.
Staikova and K. Gribachev gave in \cite{4} a classification of the
Riemannian almost product manifolds with $\tr P = 0$. In this case
the manifold $M$ is even-dimensional.

    The geometry of a Riemannian almost product manifold $(M, P, g)$
    is a geometry of both structures $g$ and $P$.
    There are important in this geometry the linear connections
    in respect of which the parallel transport determine an isomorphism
    of the tangent spaces with the structures $g$ and $P$.
    This is valid if and only if these structures are parallel with respect
    to such a connection. In the general case on a Riemannian almost product
    manifold there are countless number linear connections regarding
    which $g$ and $P$ are parallel. Such connections are called natural in \cite{5}.

    In the present work we consider the class $\W_1$ of
conformal Riemannian $P$-manifolds (shortly $\W_1$-manifolds) from
the Staikova-Gribachev classification. It is valid $\W_1 =
\overline{\W}_3\oplus\overline{\W}_6$, were $\overline{\W}_3$ and
$\overline{\W}_6$ are basic classes from the Naveira
classification. The class $\W_1$ is the largest class of
Riemannian almost product manifolds, which is closed with respect
to the group of the conformal transformations of the Riemannian
metric. This class is an analogue of the class of conformal
K\"ahler manifolds in almost Hermitian geometry.

    The main aim of the present study is to obtain properties of a $\W_1$-manifold
    admissive a natural connection whose curvature tensor is a Riemannian $P$-tensor.
    The notion a Riemannian $P$-tensor, introduced in \cite{7}
    on a Riemannian almost product manifold is an analogue of the notion
    of a K\"ahler tensor in Hermitian geometry.
    Natural connections whose curvature tensor is of K\"ahler type
    on almost contact B-metric manifolds are studied in
    \cite{ManIv40}.

    The paper is organized as follows.
    In Sec.~2 we recall necessary facts about the Riemannian almost product manifolds,
    the class $\W_1$, Riemannian $P$-tensors, natural connections.
    In Sec.~3 we obtain some properties of Riemannian $P$-tensors on
    a Riemannian almost product manifold. We obtain a presentation
    of such tensor on a 4-dimensional manifold by its scalar curvatures
    and we establish that any 4-dimensional Riemannian almost product
    manifolds is an almost Einstein manifold with respect to a Riemannian $P$-tensor.
    In Sec.~4 we consider $\W_1$-manifolds admissive a natural
    connection $\n'$
    whose curvature tensor $R'$ is a Riemannian $P$-tensor.
    The main result here is \thmref{thm-4.1} where the associated 1-forms of
    manifold are expressed by the scalar curvatures of $R'$.
    We obtain properties of the manifold for the different cases
    from the classification of such connections given in \cite{14}.
    We also discuss the cases when the manifold with such connection
    belongs to the Naveira classes $\overline{\W}_3$ and $\overline{\W}_6$.
    In Sec.~5 we consider two cases
    of a 4-dimensional manifold with connections whose curvature tensors
    are Riemannian $P$-tensors and we obtain explicit
    expressions
    of the curvature tensor of the Levi-Civita connection in these cases.


\section{Preliminaries}

Let $(M,P,g)$ be a \emph{Riemannian almost product manifold}, \ie
a differentiable manifold $M$ with a tensor field $P$ of type
$(1,1)$ and a Riemannian metric $g$ such that $P^2x=x$,
$g(Px,Py)=g(x,y)$ for any $x$, $y$ of the algebra $\X(M)$ of the
smooth vector fields on $M$. Further $x,y,z,u,w$ will stand for
arbitrary elements of $\X(M)$ or vectors in the tangent space
$T_cM$ at $c\in M$.

In this work we consider manifolds $(M, P, g)$ with $\tr{P}=0$. In
this case $M$ is an even-dimensional manifold. We assume that
$\dim{M}=2n$.

In \cite{2} A.M.~Naveira gives a classification of Riemannian
almost product manifolds with respect to the tensor $F$ of type
(0,3), defined by $F(x,y,z)=g\left(\left(\nabla_x
P\right)y,z\right), $ where $\n$ is the Levi-Civita connection of
$g$. The tensor $F$ has the properties:
\begin{equation*}
    F(x,y,z)=F(x,z,y)=-F(x,Py,Pz),\qquad
    F(x,y,Pz)=-F(x,Py,z).
\end{equation*}

Using the Naveira classification, in \cite{4} M.~Staikova and
K.~Gribachev give a classification of Riemannian almost product
manifolds $(M,P,g)$ with $\tr P=0$. The basic classes of this
classification are $\W_1$, $\W_2$ and $\W_3$. Their intersection
is the class $\W_0$ of the \emph{Riemannian $P$-manifolds}
(\cite{6}), determined by the condition $F=0$. This class is an
analogue of the class of K\"ahler manifolds in the geometry of
almost Hermitian manifolds.

The class $\W_1$ from the Staikova-Gribachev classification
contains the manifolds which are locally conformal equivalent to
Riemannian $P$-manifolds. This class plays a similar role of the
role of the class of the conformal K\"ahler manifolds in almost
Hermitian geometry. We will say that a manifold from the class
$\W_1$ is a \emph{$\W_1$-manifold}.

The characteristic condition for the class $\W_1$ is the following
\begin{equation*}\label{2}
\begin{array}{l}
\W_1: F(x,y,z)=\frac{1}{2n}\big\{ g(x,y)\ta (z)-g(x,Py)\ta (Pz)
 \big.\\[4pt]
 \phantom{\W_1: F(x,y,z)=\frac{1}{2n}} +g(x,z)\ta (y)-g(x,Pz)\ta (Py)\big\},
\end{array}
\end{equation*}
where the associated 1-form $\ta$ is determined by $
\ta(x)=g^{ij}F(e_i,e_j,x). $ Here $g^{ij}$ will stand for the
components of the inverse matrix of $g$ with respect to a basis
$\{e_i\}$ of $T_cM$ at $c\in M$. The 1-form $\ta$ is
\emph{closed}, \ie $\D\ta=0$, if and only if
$\left(\n_x\ta\right)y=\left(\n_y\ta\right)x$. Moreover, $\ta\circ
P$ is a closed 1-form if and only if
$\left(\n_x\ta\right)Py=\left(\n_y\ta\right)Px$.

In \cite{4} it is proved that
$\W_1=\overline\W_3\oplus\overline\W_6$, where $\overline\W_3$ and
$\overline\W_6$ are the classes from the Naveira classification
determined by the following conditions:
\[
\begin{array}{rl}
\overline\W_3:& \quad F(A,B,\xi)=\frac{1}{n}g(A,B)\ta^v(\xi),\quad
F(\xi,\eta,A)=0,
\\[4pt]
\overline\W_6:& \quad
F(\xi,\eta,A)=\frac{1}{n}g(\xi,\eta)\ta^h(A),\quad F(A,B,\xi)=0,
\end{array}
\]
where $A,B,\xi,\eta\in\X(M)$, $PA=A$, $PB=B$, $P\xi=-\xi$,
$P\eta=-\eta$, $\ta^v(x)=\frac{1}{2}\left(\ta(x)-\ta(Px)\right)$,
$\ta^h(x)=\frac{1}{2}\left(\ta(x)+\ta(Px)\right)$. In the case
when $\tr P=0$, the above conditions for $\overline\W_3$ and
$\overline\W_6$ can be written for any $x,y,z$ in the following
form:
\begin{equation*}\label{4'}
\begin{array}{rl}
    \overline\W_3: \quad
    &F(x,y,z)=\frac{1}{2n}\bigl\{\left[g(x,y)+g(x,Py)\right]\ta(z)\\[4pt]
    &+\left[g(x,z)+g(x,Pz)\right]\ta(y)\bigr\},\quad
    \ta(Px)=-\ta(x),
\end{array}
\end{equation*}
\begin{equation*}\label{5'}
\begin{array}{rl}
    \overline\W_6: \quad
    &F(x,y,z)=\frac{1}{2n}\bigl\{\left[g(x,y)-g(x,Py)\right]\ta(z)\\[4pt]
    &+\left[g(x,z)-g(x,Pz)\right]\ta(y)\bigr\},\quad
    \ta(Px)=\ta(x).
\end{array}
\end{equation*}

In \cite{4}, a tensor $L$ of type (0,4) with pro\-per\-ties%
\begin{equation}\label{2.4}
L(x,y,z,w)=-L(y,x,z,w)=-L(x,y,w,z),
\end{equation}
\begin{equation}\label{2.5}
L(x,y,z,w)+L(y,z,x,w)+L(z,x,y,w)=0
\end{equation}
is called a \emph{curvature-like tensor}. Such a tensor on a
Riemannian almost product manifold $(M,P,g)$ with the property
\begin{equation}\label{2.6}
L(x,y,Pz,Pw)=L(x,y,z,w)
\end{equation}
is called a \emph{Riemannian $P$-tensor} in \cite{7}. This notion
is an analogue of the notion of a K\"ahler tensor in Hermitian
geometry.

Let $S$ be a (0,2)-tensor on a Riemannian almost product manifold.
In \cite{4} it is proved that
\[
\begin{split}
\psi_1(S)(x,y,z,w)
&=g(y,z)S(x,w)-g(x,z)S(y,w)\\[4pt]
&+S(y,z)g(x,w)-S(x,z)g(y,w)
\end{split}
\]
is a curvature-like  tensor if and only if $S(x,y)=S(y,x)$, and
the tensor $$\psi_2(S)(x,y,z,w)=\psi_1(S)(x,y,Pz,Pw)$$ is
curvature-like if and only if $S(x,Py)=S(y,Px)$. Obviously
$$\psi_2(S)(x,y,Pz,Pw)=\psi_1(S)(x,y,z,w).$$ The tensors
\[
\pi_1=\frac{1}{2}\psi_1(g),\qquad
\pi_2=\frac{1}{2}\psi_2(g),\qquad
\pi_3=\psi_1(\widetilde{g})=\psi_2(\widetilde{g})
\]
are curvature-like, and the tensors $\pi_1+\pi_2$, $\pi_3$ are
Riemannian $P$-tensors.

 The curvature tensor $R$ of $\n$ is determined by
$R(x,y)z=\nabla_x \nabla_y z - \nabla_y \nabla_x z -
    \nabla_{[x,y]}z$ and the corresponding tensor of type (0,4) is defined as
follows $R(x,y,z,w)=g(R(x,y)z,w)$. We denote the Ricci tensor and
the scalar curvature of $R$ by $\rho$ and $\tau$, respectively,
\ie $\rho(y,z)=g^{ij}R(e_i,y,z,e_j)$ and
$\tau=g^{ij}\rho(e_i,e_j)$. The associated Ricci tensor $\rho^*$
and the associated scalar curvature $\tau^*$ of $R$ are determined
by $\rho^*(y,z)=g^{ij}R(e_i,y,z,Pe_j)$ and
$\tau^*=g^{ij}\rho^*(e_i,e_j)$. In a similar way there are
determined the Ricci tensor $\rho(L)$ and the scalar curvature
$\tau(L)$ for any curvature-like tensor $L$ as well as the
associated quantities $\rho^*(L)$ and $\tau^*(L)$.

In \cite{5}, a linear connection $\n'$ on a Riemannian almost
product manifold $(M,P,g)$ is called a \emph{natural connection}
if $\n' P=\n' g=0$.

In \cite{9} it is established that the natural connections $\n'$
on a $\W_1$-manifold $(M,P,g)$ form a 2-parametric family, where
the torsion $T$ of $\n'$  is determined by
\begin{equation}\label{2.9}
\begin{split}
    &T(x,y,z)=\frac{1}{2n}\left\{g(y,z)\ta(Px)-g(x,z)\ta(Py)\right\}\\[4pt]
            &+\lm\left\{g(y,z)\ta(x)-g(x,z)\ta(y)+g(y,Pz)\ta(Px)-g(x,Pz)\ta(Py)\right\}\\[4pt]
            &+\mu\left\{g(y,Pz)\ta(x)-g(x,Pz)\ta(y)+g(y,z)\ta(Px)-g(x,z)\ta(Py)\right\},
\end{split}
\end{equation}
where $\lm, \mu \in \R$.

Let us recall the following statements.

\begin{thm}[\cite{14}]\label{thm-2.1}
Let $R'$ is the curvature tensor of a natural connection $\n'$ on
a $\W_1$-manifold $(M,P,g)$. Then the following relation is valid:
\begin{equation*}\label{3.9}
    R=R'-g(p,p)\pi_1-g(q,q)\pi_2-g(p,q)\pi_3-\psi_1(S')-\psi_2(S''),
\end{equation*}
where
\begin{equation}\label{3.6}
\begin{array}{l}
    p=\lm\Omega+\left(\mu+\frac{1}{2n}\right)P\Omega,\quad q=\lm
P\Omega+\mu\Omega,\quad g(\Omega,x)=\theta(x),
\end{array}
\end{equation}
\begin{equation}\label{3.6'}
\begin{array}{rl}
    &S'(y,z)=\lm\left(\n'_y\ta\right)z+\left(\mu+\frac{1}{2n}\right)\left(\n'_y\ta\right)Pz
   \\[4pt]
   &\phantom{S'(y,z)=}-\frac{1}{2n}\left\{\lm\ta(y)\ta(Pz)+\mu\ta(y)\ta(z)\right\},
\end{array}
\end{equation}
\begin{equation}\label{3.6''}
\begin{array}{rl}
    &S''(y,z)=\lm\left(\n'_y\ta\right)z+\mu\left(\n'_y\ta\right)Pz\\[4pt]
    &\phantom{S''(y,z)=}+\frac{1}{2n}\left\{\lm\ta(Py)\ta(z)+\mu\ta(Py)\ta(Pz)\right\}.
\end{array}
\end{equation}
\end{thm}

\begin{thm}[\cite{14}]\label{thm-2.2}
Let $R'$ be the curvature tensor of a natural connection $\n'$ on
a $\W_1$-manifold
$(M,P,g)\notin\overline{\W}_3\cup\overline{\W}_6$. Then the all
possible cases are as follows:
\begin{enumerate}
    \item[i)] If $\n'$ is the connection $D$ determined by
    $\lm=\mu=0$, then $R'$ is a Riemannian $P$-tensor if and only
    if the 1-form $\ta$ is not closed and the 1-form $\ta\circ P$ is
    closed;
    \item[ii)] If $\n'$ is the connection $\widetilde{D}$ determined by
    $\lm=0$, $\mu=-\frac{1}{2n}$, then $R'$ is a Riemannian $P$-tensor if and only
    if the 1-form $\ta$ is closed and the 1-form $\ta\circ P$ is
    not closed;
    \item[iii)] If $\n'$ is a connection for which
    $\lm^2-\mu^2-\frac{\mu}{2n}\neq 0$, then $R'$ is a Riemannian $P$-tensor if and only
    if the 1-forms $\ta$ and $\ta\circ P$ are closed;
    \item[iv)] If $\n'$ is a connection for which $\lm\neq 0$,
    $\lm^2-\mu^2-\frac{\mu}{2n}=0$ and $R'$ is a Riemannian $P$-tensor then the 1-forms $\ta$ and
    $\ta\circ P$ are not closed.
\end{enumerate}
\end{thm}

Let us remark that the connection $D$ determined by $\lm=\mu=0$ is
investigated in \cite{8}.


\section{Some properties of the Riemannian $P$-tensors
on Riemannian almost product manifolds}

\begin{lem}\label{lem-3.1}
Let $L$ be a Riemannian $P$-tensor on a Riemannian almost
product manifold $(M,P,g)$. Then the following properties are
valid:
\begin{equation}\label{3.12}
    L(x,Py,Pz,w)=L(Px,Py,z,w)=L(x,y,z,w),
\end{equation}
\begin{equation}\label{3.13}
    L(Px,y,z,w)=L(x,Py,z,w)=L(x,y,Pz,w)=L(x,y,z,Pw).
\end{equation}
\end{lem}
\begin{proof}
Equalities \eqref{2.4}, \eqref{2.5} and \eqref{2.6} imply
\begin{equation}\label{3.13'}
    L(Px,Py,z,w)=L(x,y,z,w),
\end{equation}
and the following equality follows from \eqref{2.5} and
\eqref{3.13'}:
\[
    L(x,Py,Pz,w)=-L(y,z,x,w)-L(z,Px,Py,w).
\]
In the latter equality, we substitute $Px$ and $Py$ for $x$ and
$y$, respectively. Then, according to \eqref{3.13'}, we obtain
\[
    L(x,Py,Pz,w)=-L(Py,z,Px,w)-L(z,x,y,w).
\]
The summation of the latter two equalities, bearing in mind
\eqref{2.5} and \eqref{2.6}, implies
$
    L(x,Py,Pz,w)=L(x,y,z,w).
$
Equalities \eqref{3.12} follow from the last equality and
\eqref{3.13'} as well as \eqref{3.13} --- from \eqref{3.12} and
\eqref{2.6}.
\end{proof}

\begin{thm}\label{thm-3.2}
A curvature-like tensor $L$ on 4-dimensional Riemannian almost
product manifold is a Riemannian $P$-tensor if and only if $L$ has
the following form:
\begin{equation}\label{3.14}
    L=\frac{1}{8}\left\{\tau(L)(\pi_1+\pi_2)+\tau^*(L)\pi_3\right\}.
\end{equation}
\end{thm}

\begin{proof}
Let $L$ be a Riemannian $P$-tensor on a 4-dimensional Riemannian
almost product manifold $(M,P,g)$ and
$\left\{E_1,E_2,PE_1,PE_2\right\}$ be an orthonormal adapted basis
(\cite{15}) of $T_cM$, $c\in M$. Taking into account \eqref{2.6}
and \lemref{lem-3.1}, we obtain that the non-zero components of
$L$ are expressed by the sectional curvature
$\nu=L(E_1,E_2,E_1,E_2)$ of the 2-plane $\left\{E_1,E_2\right\}$
and its associated sectional curvature
$\nu^*=L(E_1,E_2,E_1,PE_2)$.

Let the arbitrary vectors $x,y,z,w$ have the form
\begin{equation*}\label{3.16}
\begin{array}{l}
    x=x^1E_1+x^2E_2+x^3PE_1+x^4PE_2,\\[4pt]
    y=y^1E_1+y^2E_2+y^3PE_1+y^4PE_2,\\[4pt]
    z=z^1E_1+z^2E_2+z^3PE_1+z^4PE_2,\\[4pt]
    w=w^1E_1+w^2E_2+w^3PE_1+w^4PE_2.
\end{array}
\end{equation*}

Taking into account the linearity of $L$, equalities \eqref{2.4},
\eqref{2.6} and \lemref{lem-3.1}, we get
\begin{equation}\label{3.17}
\begin{split}
    L(x,y,z,w)&=\nu\left\{a(x,y)a(z,w)+b(x,y)b(z,w)\right\}\\[4pt]
    &+\nu^*\left\{a(x,y)b(z,w)+b(x,y)a(z,w)\right\},
\end{split}
\end{equation}
where
\[
\begin{array}{l}
    a(x,y)=x^1y^2+x^3y^4-x^2y^1-x^4y^3,\\[4pt]
    b(x,y)=x^1y^4+x^3y^2-x^2y^3-x^4y^1.
\end{array}
\]
Since the basis $\left\{E_1,E_2,PE_1,PE_2\right\}$ is orthonormal,
the following equalities are valid
\[
\begin{array}{l}
    (\pi_1+\pi_2)(x,y,z,w)=-a(x,y)a(z,w)-b(x,y)b(z,w),\\[4pt]
    \pi_3(x,y,z,w)=-a(x,y)b(z,w)-b(x,y)a(z,w)
\end{array}
\]
and then \eqref{3.17} takes the form
\begin{equation}\label{3.18}
    L(x,y,z,w)=-\nu(\pi_1+\pi_2)(x,y,z,w)-\nu^*\pi_3(x,y,z,w).
\end{equation}
Equality \eqref{3.18} implies the following form of the Ricci
tensor of $L$:
\begin{equation}\label{3.19}
    \rho(L)(y,z)=-2\nu g(y,z)-2\nu^* g(y,Pz).
\end{equation}
Then we obtain the following formulae for the scalar curvatures of
$L$:
\[
    \tau(L)=-8\nu,\qquad \tau^*(L)=-8\nu^*.
\]
Thus, \eqref{3.18} implies \eqref{3.14}.

Vice versa, let $L$ be a curvature-like tensor of the form
\eqref{3.14}. Then, according to $\pi_1+\pi_2$ and $\pi_3$ are
Riemannian $P$-tensors, the tensor $L$ is also a Riemannian
$P$-tensor.
\end{proof}

Bearing in mind \cite{6}, a 2-plane $\al\in T_c M$ of a Riemannian
almost product manifold $(M,P,g)$ is called totally real
(respectively, invariant) if $\al$ and $P\al$ are orthogonal
(respectively, $\al$ and $P\al$ coincide). By \lemref{lem-3.1} we
establish that  all totally real basic 2-planes in the orthonormal
adapted basis $\left\{E_1,E_2,PE_1,PE_2\right\}$ have a sectional
curvature $\nu=L(E_1,E_2,E_1,E_2)$ with respect to the Riemannian
$P$-tensor $L$ and the invariant 2-planes  have zero sectional
curvatures with respect to $L$.

Since equality \eqref{3.19} for a Riemannian $P$-tensor $L$
implies
\[
    \rho(L)(y,z)=\frac{1}{4}\left\{\tau(L) g(y,z)+\tau^*(L)
    g(y,Pz)\right\},
\]
the manifold $(M,P,g)$ is almost Einstein with respect to  $L$.

So, the following statement is valid.

\begin{prop}\label{prop-3.3}
Each 4-dimensional Riemannian almost product manifold is almost
Einstein, which is of point-wise constant totally real sectional
curvatures and zero invariant sectional curvatures with respect to
arbitrary Riemannian $P$-tensor.
\end{prop}


\section{$\W_1$-manifolds with a natural connection whose curvature tensor
is a Riemannian $P$-tensor}

It is known (\cite{16}), that the curvature tensor of a linear
connection $\n'$ with torsion $T$ satisfies the second Bianchi
identity
\begin{equation}\label{4.4*}
   \sx
\left\{\left(\n'_x
R'\right)(y,z,u)+R'\left(T(x,y),z\right),u\right\}=0,
\end{equation}
where $\sx$ is stand for the cyclic sum by $x,y,z$.

Now, let $\n'$ be a natural connection on $\W_1$-manifold
$(M,P,g)$. Because $\n'g=0$, equality \eqref{4.4*} implies
\begin{equation}\label{4.4}
   \sx
\left\{\left(\n'_x
R'\right)(y,z,u,w)+R'\left(T(x,y),z,u,w\right)\right\}=0.
\end{equation}

Using \eqref{2.9} for the torsion tensor $T$ of type (1,2) of
$\n'$, we have
\begin{equation*}\label{4.1}
\begin{split}
    T(x,y)&=\left\{\lm\ta(x)+\left(\mu+\frac{1}{2n}\right)\ta(Px)\right\}y+\left[\lm\ta(Px)+\mu\ta(x)\right]Py\\[4pt]
           &
           -\left\{\lm\ta(y)+\left(\mu+\frac{1}{2n}\right)\ta(Py)\right\}x-\left[\lm\ta(Py)+\mu\ta(y)\right]Px,
\end{split}
\end{equation*}
which implies immediately
\begin{equation}\label{4.3}
   T(x,y)-PT(Px,y)=\frac{1}{2n}\left\{\ta(Px)y-\ta(x)Py\right\}.
\end{equation}

Let the curvature tensor $R'$ of $\n'$ be a Riemannian $P$-tensor.

In \eqref{4.4}, we substitute $Px$ and $Pw$ for $x$ and $w$,
respectively. Then, we subtract the result equality from
\eqref{4.4}. Using \lemref{lem-3.1} for $R'$, we get
\[
\begin{split}
&\left(\n'_x R'\right)(y,z,u,w)+R'\left(T(x,y)-PT(Px,y),z,u,w\right)\\[4pt]
&-\left(\n'_{Px} R'\right)(y,z,u,Pw)
-R'\left(T(x,z)-PT(Px,z),y,u,w\right)=0.
\end{split}
\]

Hence, applying \lemref{lem-3.1}, the property \eqref{4.3} and the
linearity of $R'$, we obtain
\begin{equation}\label{4.6}
\begin{split}
&\left(\n'_x R'\right)(y,z,u,w)-\left(\n'_{Px} R'\right)(y,z,u,Pw)\\[4pt]
&+\frac{1}{n}\left\{\ta(Px)R'(y,z,u,w)-\ta(x)R'(y,Pz,u,w)\right\}=0.
\end{split}
\end{equation}

Let $R'$ have a Ricci tensor $\rho'$ and scalar curvatures $\tau'$
and $\tau^{*}{'}$.

Applying a contraction by $g^{ij}$ to $y=e_i$ and $w=e_j$, from
\eqref{4.6} we get
\[
\begin{split}
&\left(\n'_x \rho'\right)(z,u)-\left(\n'_{Px} \rho'\right)(z,Pu)
+\frac{1}{n}\left\{\ta(Px)\rho'(z,u)-\ta(x)\rho'(Pz,u)\right\}=0.
\end{split}
\]
After that,  applying a contraction by $g^{ks}$ to $z=e_k$ and
$u=e_s$, we have
\begin{equation}\label{4.8}
    \D\tau'(x)-\D\tau^{*}{'}(Px)+\frac{1}{n}
    \left\{\ta(Px)\tau'-\ta(x)\tau^{*}{'}\right\}=0,
\end{equation}
which implies
\begin{equation}\label{4.9}
    \D\tau'(Px)-\D\tau^{*}{'}(x)+\frac{1}{n}
    \left\{\ta(x)\tau'-\ta(Px)\tau^{*}{'}\right\}=0.
\end{equation}

Equalities \eqref{4.8} and \eqref{4.9} determine a linear system
for the 1-forms $\ta(x)$ and $\ta(Px)$ with a determinant
$\Delta=\left(\tau^{*}{'}\right)^2-\left(\tau{'}\right)^2$.


\subsection{Case $(M,P,g)\notin\overline{\W}_3\cup\overline{\W}_6$}

In this case, for the considered $\W_1$-manifold is valid
$\ta(x)\neq \pm \ta(Px)$.

Let $\Delta\neq 0$, i.e.
$\left|\tau^{*}{'}\right|\neq\left|\tau{'}\right|$. Then the
linear system has the following solution:
\begin{equation}\label{4.10}
\begin{split}
&\ta(x)=\frac{n}{2}\left\{\D\ln \left|\frac{\tau^{*}{'}+\tau{'}}{\tau^{*}{'}-\tau{'}}\right|(x)
-\D\ln \left|\left(\tau^{*}{'}\right)^2-\left(\tau{'}\right)^2\right|(Px)\right\},\\[4pt]
&\ta(Px)=\frac{n}{2}\left\{\D\ln
\left|\frac{\tau^{*}{'}+\tau{'}}{\tau^{*}{'}-\tau{'}}\right|(Px)
-\D\ln
\left|\left(\tau^{*}{'}\right)^2-\left(\tau{'}\right)^2\right|(x)\right\}.
\end{split}
\end{equation}

Let $\Delta=0$, i.e.
$\left|\tau^{*}{'}\right|=\left|\tau{'}\right|$. Then, from
\eqref{4.8} we have
\[
\tau{'}\left\{\ta(Px)-\ep\ta(x)\right\}=n\left\{\ep\D\tau'(x)-\D\ta'(Px)\right\},\qquad
\ep=\pm 1.
\]
Thus, for $\tau{'}\neq 0$ the following equalities are valid:
\begin{equation}\label{4.11}
\begin{split}
&\ta(Px)-\ta(x)=n\left\{\D\ln\tau{'}(x)-\D\ln\tau{'}(Px)\right\}
\;\; \text{for}\;\;\tau{'}=\tau^{*}{'}\neq 0,\\[4pt]
&\ta(Px)+\ta(x)=-n\left\{\D\ln\tau{'}(x)+\D\ln\tau{'}(Px)\right\}
\;\; \text{for}\;\;\tau{'}=-\tau^{*}{'}\neq 0.
\end{split}
\end{equation}

Hence we establish the truthfulness of the following
\begin{thm}\label{thm-4.1}
Let the curvature tensor $R'$ of a natural connection $\n'$ on a
$\W_1$-manifold $(M,P,g)\notin\overline{\W}_3\cup\overline{\W}_6$
is a Riemannian $P$-tensor with scalar curvatures $\tau{'}$ and
$\tau^{*}{'}$. Then for the 1-forms $\ta$ and  $\ta\circ P$ are
valid equalities \eqref{4.10} for
$\left|\tau^{*}{'}\right|\neq\left|\tau{'}\right|$ and equalities
\eqref{4.11} for
$\left|\tau^{*}{'}\right|=\left|\tau{'}\right|\neq 0$.
\end{thm}

From \thmref{thm-2.2} and \thmref{thm-4.1} we obtain the following
\begin{thm}\label{thm-4.2}
Let the curvature tensor $R'$ of a natural connection $\n'$ on a
$\W_1$-manifold $(M,P,g)\notin\overline{\W}_3\cup\overline{\W}_6$
is a Riemannian $P$-tensor with scalar curvatures $\tau{'}$ and
$\tau^{*}{'}$. Then
\begin{enumerate}
    \item[i)] For the connection $\n'=D$ determined by $\lm=\mu=0$,
    are valid the properties:\\
    a) If $\left|\tau^{*}{'}\right|\neq\left|\tau{'}\right|$, then
    $\D\ln
    \left|\frac{\tau^{*}{'}+\tau{'}}{\tau^{*}{'}-\tau{'}}\right|\circ
    P$ is a closed 1-form;\\
    b) If $\left|\tau^{*}{'}\right|=\left|\tau{'}\right|\neq 0$, i.e.
    $\tau^{*}{'}=\ep \tau{'}\neq 0$ $(\ep=\pm 1)$, then we have $\D\ta=\ep n
    \D\left(\D\ln\tau'\circ P\right)$. %
    \item[ii)] For the connection $\n'=\widetilde{D}$ determined by $\lm=0$ and
    $\mu=-\frac{1}{2n}$,
    are valid the properties:\\
    a) If $\left|\tau^{*}{'}\right|\neq\left|\tau{'}\right|$, then
    $\D\ln
    \left|\left(\tau^{*}{'}\right)^2-\left(\tau{'}\right)^2\right|\circ
    P$ is a closed 1-form;\\
    b) If $\left|\tau^{*}{'}\right|=\left|\tau{'}\right|\neq 0$, i.e.
    $\tau^{*}{'}=\ep \tau{'}\neq 0$ $(\ep=\pm 1)$, then we have $\D\ta\circ P=\ep n
    \D\left(\D\ln\tau'\circ P\right)$. %
    \item[iii)] For the connections $\n'$ determined by
    $\lm^2-\mu^2-\frac{\mu}{2n}=0$,
    are valid the properties:\\
    a) If $\left|\tau^{*}{'}\right|\neq\left|\tau{'}\right|$, then
    $\D\ln
    \left|\frac{\tau^{*}{'}+\tau{'}}{\tau^{*}{'}-\tau{'}}\right|\circ
    P$ and $\D\ln
    \left|\left(\tau^{*}{'}\right)^2-\left(\tau{'}\right)^2\right|\allowbreak{}\circ
    P$ are closed 1-forms;\\
    b) If $\left|\tau^{*}{'}\right|=\left|\tau{'}\right|\neq 0$,
    then $\D\ln\tau'\circ P$ is a closed 1-form. %
\end{enumerate}
\end{thm}
\begin{proof}
Let $\n'$ is the connection $D$ determined by $\lm=\mu=0$. Then,
according to \thmref{thm-2.2}, the 1-form $\ta\circ P$ is closed.
If $\left|\tau^{*}{'}\right|\neq\left|\tau{'}\right|$, then by
virtue of \thmref{thm-4.1} are valid equalities \eqref{4.10}. We
differentiate the second equality of \eqref{4.10} and because of
$\D\ta\circ P=0$ and $\D\circ\D=0$, we obtain $\D\left(\D\ln
    \left|\frac{\tau^{*}{'}+\tau{'}}{\tau^{*}{'}-\tau{'}}\right|\circ
    P\right)=0$, i.e. the property a) from i) is valid.
If $\left|\tau^{*}{'}\right|=\left|\tau{'}\right|\neq 0$, then by
virtue of \thmref{thm-4.1} are valid equalities \eqref{4.11}.
After that by a differentiation we have $\D\ta=\ep
n\D\left(\D\ln\tau'\circ P\right)$, i.e. the property b) from i)
is valid.

In a similar way we establish the properties from ii) and iii).
\end{proof}


\subsection{Case $(M,P,g)\in\overline{\W}_3\cup\overline{\W}_6$}

In this case, for the considered $\W_1$-manifold is valid $\ta(x)=
\pm \ta(Px)$.
\begin{thm}\label{thm-4.3}
Let the curvature tensor $R'$ of a natural connection $\n'$ on a
$\W_1$-manifold $(M,P,g)\in\overline{\W}_3\cup\overline{\W}_6$ is
a Riemannian $P$-tensor with scalar curvatures $\tau{'}$ and
$\tau^{*}{'}$. Then
\begin{enumerate}
    \item[i)] For $(M,P,g)\in\overline{\W}_3$ the following properties are valid:
    \\
    a) If $\left|\tau^{*}{'}\right|\neq\left|\tau{'}\right|$, then
    $\D\left(\tau^{*}{'}-\tau{'}\right)\circ
    P$ is a closed 1-form and the 1-form $\ta(x)$ has the form
    \[
    \ta(x)=\frac{n}{2}\left\{\D\ln
    \left|\tau^{*}{'}+\tau{'}\right|(x)
    -\D\ln
    \left|\tau^{*}{'}+\tau{'}\right|(Px)\right\};
    \]
    b) If $\tau^{*}{'}=\tau{'}\neq 0$, then
    the 1-form $\ta(x)$ has the form
    \[
    \ta(x)=\frac{n}{2}\left\{\D\ln
    \left|\tau{'}\right|(x)
    -\D\ln
    \left|\tau{'}\right|(Px)\right\},
    \]
    and if $\tau^{*}{'}=-\tau{'}\neq 0$, then $\D\tau'\circ P$ is a closed
    1-form.
    \item[ii)] For $(M,P,g)\in\overline{\W}_6$ the following properties are valid:
    \\
    a) If $\left|\tau^{*}{'}\right|\neq\left|\tau{'}\right|$, then
    $\D\left(\tau^{*}{'}+\tau{'}\right)\circ
    P$ is a closed 1-form and the 1-form $\ta(x)$ has the form
    \[
    \ta(x)=-\frac{n}{2}\left\{\D\ln
    \left|\tau^{*}{'}-\tau{'}\right|(x)
    +\D\ln
    \left|\tau^{*}{'}-\tau{'}\right|(Px)\right\};
    \]
    b) If $\left|\tau^{*}{'}\right|=\left|\tau{'}\right|\neq 0$, then
    $\D\tau{'}\circ
    P$ is a closed 1-form
    and if $\tau^{*}{'}=-\tau{'}\neq 0$, then 1-form $\ta(x)$ has the form
    \[
    \ta(x)=-\frac{n}{2}\left\{\D\ln
    \left|\tau{'}\right|(x)
    +\D\ln
    \left|\tau{'}\right|(Px)\right\}.
    \]
\end{enumerate}
\end{thm}
\begin{proof}
Let $(M,P,g)\in\overline{\W}_3$, i.e. $\ta(Px)=-\ta(x)$. Then,
according to \eqref{4.9}, it is follows
\begin{equation}\label{4.12}
    \ta(x)\left(\tau^{*}{'}+\tau{'}\right)=n\left\{\D\tau^{*}{'}(x)-\D
    \tau{'}(Px)\right\}.
\end{equation}
Hence, for $\left|\tau^{*}{'}\right|\neq\left|\tau{'}\right|$ we
obtain the equalities
\[
\ta(x)=n\frac{\D\tau^{*}{'}(x)-\D
    \tau{'}(Px)}{\tau^{*}{'}+\tau{'}},\qquad
\ta(x)=-n\frac{\D\tau^{*}{'}(Px)-\D
    \tau{'}(x)}{\tau^{*}{'}+\tau{'}},
\]
which imply the property a) from i).

If $\left|\tau^{*}{'}\right|=\left|\tau{'}\right|\neq 0$, then
    \eqref{4.12} implies the property b) from i).

In a similar way, if $(M,P,g)\in\overline{\W}_6$, i.e.
$\ta(Px)=\ta(x)$, then \eqref{4.9} implies the property a) and b)
from ii).
\end{proof}


\section{Two cases of 4-dimensional $\W_1$-manifolds}

According to \thmref{thm-2.1} and \thmref{thm-3.2}, we obtain the
following
\begin{thm}\label{thm-5.1}
Let the curvature tensor $R'$ of a natural connection $\n'$ on a
4-dimensional $\W_1$-manifold $(M,P,g)$ is a Riemannian $P$-tensor
with scalar curvatures $\tau{'}$ and $\tau^{*}{'}$. Then the
curvature tensor $R$ of the Levi-Civita connection has the
following form
\begin{equation}\label{5.1}
\begin{split}
    R&=\frac{1}{8}\left\{\tau{'}(\pi_1+\pi_2)+\tau^{*}{'}\pi_3\right\}-\psi_1(S')-\psi_2(S'')\\[4pt]
        &-g(p,p)\pi_1-g(q,q)\pi_2-g(p,q)\pi_3.
\end{split}
\end{equation}
\end{thm}

Further we consider the cases when $\n'$ is the connection $D$ or
the connection $\widetilde{D}$.

\subsection{Case $\n'=D$}
Let $\n'$ be the natural connection $D$ on a 4-dimensional
$\W_1$-manifold $(M,P,g)$ whose curvature tensor $R'$ is a
Riemannian $P$-tensor. Taking into account that  $\lm=\mu=0$ and
$n=2$, from \eqref{3.6}, \eqref{3.6'} and \eqref{3.6''} we have
\[
g(p,p)=\frac{\ta(\Omega)}{16},\;\; g(q,q)=g(p,q)=0,\;\;
S'(y,z)=\frac{\left(D_y\ta\right)Pz}{4},\;\; S''(y,z)=0.
\]
Then from \eqref{5.1} we obtain the relation
\begin{equation}\label{5.2}
    R=\frac{1}{8}\left\{\tau{'}(\pi_1+\pi_2)+\tau^{*}{'}\pi_3\right\}
    -\frac{\ta(\Omega)}{16}\pi_1-\psi_1(S'),
\end{equation}
which implies immediately
\begin{equation*}\label{5.3}
    \rho=\rho'-\frac{3\ta(\Omega)}{16}g - \tr S' g -2 S',
\end{equation*}
\begin{equation}\label{5.4}
    \tau=\tau'-\frac{\ta(\Omega)}{4} - 6\tr S',\qquad
    \tau^*=\tau^{*}{'}-2\tr \widetilde{S'},
\end{equation}
where $\widetilde{S'}(y,z)=S'(y,Pz)$.

In \cite{8}, it is get the equality
\[
D_yz=\n_yz+\frac{1}{2n}\left\{g(y,z)P\Omega-\ta(Pz)y\right\}.
\]
From the latter equality, we obtain the following
\[
\left(D_y\ta\right)z=\left(\n_y\ta\right)z
+\frac{1}{4}\left\{\ta(y)\ta(Pz)-g(y,z)\ta(P\Omega)\right\}.
\]
Then $S'$ is expressed as follows:
\[
S'(y,z)=\frac{1}{4}\left(\n_y\ta\right)Pz+ \frac{1}{16} \left\{
\ta(y)\ta(z)-g(y,Pz)\ta(P\Omega)\right\}.
\]
Hence, the following equalities follows immediately:
\begin{equation}\label{5.5}
    \tr S'=\frac{\dv P\Omega}{4}+\frac{\ta(\Omega)}{16},\qquad
    \tr \widetilde{S'}=\frac{\dv\Omega}{4}-\frac{3\ta(P\Omega)}{16},
\end{equation}
where $\dv\Omega$ is stand for the divergence of $\Omega$.

The equalities \eqref{5.4} and \eqref{5.5} imply
\begin{equation}\label{5.6}
    \tau'=\tau+\frac{3\dv P\Omega}{2}+\frac{9\ta(\Omega)}{8},\qquad
    \tau^{*}{'}=\tau^*+\frac{\dv\Omega}{2}-\frac{3\ta(P\Omega)}{8}.
\end{equation}

By virtue of \eqref{5.2} and \eqref{5.6} we obtain the
truthfulness of the following
\begin{thm}\label{thm-5.2}
Let $(M,P,g)$ be a 4-dimensional $\W_1$-manifold. If $(M,P,g)$
admit the natural connection $D$ determined by $\lm=\mu=0$ whose
curvature tensor is a Riemannian $P$-tensor, then the curvature
tensor $R$ of the Levi-Civita connection has the following form
\[
\begin{split}
    R&=\frac{1}{32}\left\{4\tau+6\dv P\Omega+3\ta(\Omega)\right\}(\pi_1+\pi_2)\\[4pt]
&+\frac{1}{64}\left\{8\tau^*+4\dv
\Omega-3\ta(P\Omega)\right\}\pi_3
        -\frac{\ta(\Omega)}{16}\pi_1-\psi_1(S').
\end{split}
\]
\end{thm}

\subsection{Case $\n'=\widetilde{D}$} %
Let $\n'$ be the natural connection $\widetilde{D}$ on a
4-dimensional $\W_1$-manifold $(M,P,g)$ whose curvature tensor
$R'$ is a Riemannian $P$-tensor. Taking into account that $\lm=0$,
$\mu=-\frac{1}{4}$ and $n=2$, from \eqref{3.6}, \eqref{3.6'} and
\eqref{3.6''} we have
\begin{equation}\label{*}
\begin{split}
&g(p,p)=g(p,q)=0, \quad g(q,q)=\frac{1}{16}\ta(\Omega),\\[4pt]
&S'(y,z)=\frac{1}{16}\ta(y)\ta(z),\quad
S''(y,z)=-\frac{1}{4}\left(\widetilde{D}_y\ta\right)Pz-\frac{1}{16}\ta(Py)\ta(Pz).
\end{split}
\end{equation}

Then from \eqref{5.1} we obtain the relation
\begin{equation}\label{5.7}
    R=\frac{1}{8}\left\{\tau{'}(\pi_1+\pi_2)+\tau^{*}{'}\pi_3\right\}
    -\frac{\ta(\Omega)}{16}\pi_2-\psi_1(S')+\psi_2(S''),
\end{equation}
which implies immediately
\begin{equation*}\label{5.8}
    \rho=\rho'+\frac{\ta(\Omega)}{16}g - \tr S' g -2 S'-\tr \widetilde{S''} \widetilde{g}+2S'',
\end{equation*}
\begin{equation}\label{5.9}
    \tau=\tau'+\frac{\ta(\Omega)}{4} - 6\tr S'+2\tr S'',\quad
    \tau^*=\tau^{*}{'}-2\tr \widetilde{S'}-2\tr \widetilde{S''},
\end{equation}
where $\widetilde{g}(y,z)=g(y,Pz)$ and
$\widetilde{S''}(y,z)=S''(y,Pz)$.

According to \cite{9}, we have $\widetilde{D}=\n+Q$, where
$Q(x,y,z)=T(z,y,x)$. Then, using \eqref{2.9}, $\lm=0$,
$\mu=-\frac{1}{4}$ and $n=2$, we get
\[
\widetilde{D}_yz=\n_yz+\frac{1}{4}\left\{\ta(z)Py-g(y,Pz)\ta(\Omega)\right\},
\]
which implies
\[
\left(\widetilde{D}_y\ta\right)z=\left(\n_y\ta\right)z
-\frac{1}{4}\left\{\ta(z)\ta(Py)-g(y,Pz)\ta(\Omega)\right\}.
\]
Then we obtain the following expression of $S''$:
\[
S''(y,z)=-\frac{1}{4}\left(\n_y\ta\right)Pz-\frac{1}{16}g(y,z)\ta(\Omega).
\]
The latter equality and the first equality of \eqref{*} imply
\begin{equation}\label{5.10}
\begin{split}
    &\tr S'=\frac{\ta(\Omega)}{16},\qquad \tr \widetilde{S'}=\frac{\ta(P\Omega)}{16},\\[4pt]
    &\tr S''=-\frac{\dv P\Omega+\ta(\Omega)}{4},\qquad
    \tr \widetilde{S''}=-\frac{\dv \Omega}{4}.
\end{split}
\end{equation}

From equalities \eqref{5.9} and \eqref{5.10} we have
\begin{equation}\label{5.11}
    \tau'=\tau+\frac{\dv P\Omega}{2}+\frac{5\ta(\Omega)}{8},\qquad
    \tau^{*}{'}=\tau^*-\frac{\dv\Omega}{2}+\frac{\ta(P\Omega)}{8}.
\end{equation}

By virtue of \eqref{5.7} and \eqref{5.11} we obtain the
truthfulness of the following
\begin{thm}\label{thm-5.3}
Let $(M,P,g)$ be a 4-dimensional $\W_1$-manifold. If $(M,P,g)$
admit the natural connection $\widetilde{D}$ determined by $\lm=0$
and $\mu=-\frac{1}{4}$ whose curvature tensor is a Riemannian
$P$-tensor, then the curvature tensor $R$ of the Levi-Civita
connection has the following form
\[
\begin{split}
    R&=\frac{1}{64}\left\{8\tau+4\dv P\Omega+5\ta(\Omega)\right\}(\pi_1+\pi_2)\\[4pt]
&+\frac{1}{64}\left\{8\tau^*-4\dv \Omega+\ta(P\Omega)\right\}\pi_3
        -\frac{\ta(\Omega)}{16}\pi_2-\psi_1(S')-\psi_2(S'').
\end{split}
\]
\end{thm}


\end{document}